\documentclass[11pt]{amsart}
\usepackage{amsmath, amsfonts, amsthm, amssymb, multicol, mathtools, dsfont, verbatim}
\usepackage{graphicx}
\usepackage{float, hyperref}
\usepackage{scalerel,stackengine, subcaption}
\usepackage[usenames,dvipsnames]{xcolor}
\usepackage{enumitem}
\usepackage{pgfplots}
\usepgfplotslibrary{fillbetween}
\pgfplotsset{width=10cm,compat=1.9}
\usepackage[square,sort,comma,numbers]{natbib}
\allowdisplaybreaks

\usepackage[square,sort,comma,numbers]{natbib}
\usepackage{fancyhdr}

\makeatletter
\def\@setauthors{%
  \begingroup
  \def\thanks{\protect\thanks@warning}%
  \trivlist
  \centering\footnotesize \@topsep30\p@\relax
  \advance\@topsep by -\baselineskip
  \item\relax
  \author@andify\authors
  \def\\{\protect\linebreak}

  \normalsize\lowercase{\authors}%
  
	\ifx\@empty\contribs
  \else
    ,\penalty-3 \space \@setcontribs
    \@closetoccontribs
  \fi
  \endtrivlist
  \endgroup
}
\def\@settitle{\begin{center}
\LARGE\lowercase{\@title}
  \end{center}%
}
\makeatother

\definecolor{lightblue}{HTML}{2B77A4}
\definecolor{darkred}{HTML}{9E0D0D}
\hypersetup{
	colorlinks=true,
	linkcolor=darkred,
	urlcolor=darkred,
	citecolor=lightblue
}
\numberwithin{equation}{section}

\newtheorem{thm}{Theorem}[section]
\newtheorem{lma}[thm]{Lemma}

\renewcommand{\epsilon}{\varepsilon}
\newcommand{\eps}{\varepsilon}

\newcommand{\rd}{\mathbb{R}^d}

\renewcommand{\geq}{\geqslant}
\renewcommand{\leq}{\leqslant}

\newcommand{\ubd}{\overline{\dim}_{\textup{B}}}
\newcommand{\lbd}{\underline{\dim}_{\textup{B}}}
\newcommand{\hd}{\dim_{\textup{H}}}
\newcommand{\bd}{\dim_{\textup{B}}}
\newcommand{\pd}{\dim_{\textup{P}}}

\title{On variants of the Furstenberg set problem}

\author{Jonathan M. Fraser\\ \\
 University of St Andrews, Scotland\\
\MakeLowercase{Email: jmf32@st-andrews.ac.uk}}
\thanks{The  author was  financially supported by   an \emph{EPSRC Standard Grant} (EP/Y029550/1) and a \emph{Leverhulme Trust Research Project Grant} (RPG-2023-281).}

\begin{document}


\maketitle
\thispagestyle{empty}

\begin{abstract}
Given an integer $d \geq 2$,  $s \in (0,1]$, and $t \in [0,2(d-1)]$, suppose a  set  $X$ in $\rd$  has the following property:~there is a collection of lines of packing dimension $t$ such that every line from the collection intersects $X$ in a  set of packing dimension at least $s$.  We show that such sets must have packing dimension at least $\max\{s,t/2\}$ and that this bound is sharp.  In particular, the special case $d=2$ solves a variant of the \emph{Furstenberg set problem} for packing dimension. We also solve the upper and lower box  dimension variants of the problem. In both of these cases the sharp threshold is $\max\{s,t+1-d\}$.
\end{abstract}

\tableofcontents

\section{Motivation and main results}

Given $s \in (0,1]$ and $t \in [0,2]$, an $(s,t)$-\emph{Furstenberg set} $X$ is a subset of the plane for which there is a set of lines $\mathcal{L}$ of Hausdorff dimension at least $t$ such that for every line $L \in \mathcal{L}$, $\hd X \cap L \geq s$, where $\hd$ denotes the \emph{Hausdorff dimension}.  The celebrated \emph{Furstenberg set problem}---which goes back to Furstenberg \cite{furst}; see also Wolff \cite{wolff} and Bourgain \cite{bourgain, bourgain2}---is to determine the smallest possible  Hausdorff dimension of an  $(s,t)$-Furstenberg set as a function of $s$ and $t$.  Results pertaining to this problem have applications in a variety of directions, including exceptional set estimates for orthogonal projections and sum-product type theorems. The problem was recently resolved by Ren and Wang \cite{wang} where it was proved that an $(s,t)$-Furstenberg set   has Hausdorff dimension at least
\[
\min\left\{s+t, \frac{3s+t}{2}, s+1\right\}.
\]
In particular, this  bound is the best possible.   Their  proof built on significant progress in the area in recent years; see, for example,  \cite{GSW19,OS23,OS23+,shmerkinwang,wang} and the references therein. 

Our main question is:~what if Hausdorff dimension is replaced by a different notion of dimension throughout the above discussion?  Can similar dimension bounds be derived?  We completely resolve this problem for the upper and lower box dimensions and the packing dimension and show that the sharp bounds take on a different form; indeed, Furstenberg sets in this context can be much smaller than in the Hausdorff dimension setting. We are especially interested in the planar case ($d=2$) but we also resolve the general case in higher dimensions. 
We write $\ubd$ and $\lbd$ to denote the upper and lower box dimensions and $\pd$ to denote the packing dimension; see Section \ref{notation} for the definitions. Our first theorem gives general lower bounds for the upper and lower box dimension variant.

\begin{thm} \label{mainbound}
Let  $d \geq 2$ be an integer,  $s \in [0,1]$, and $t \in [0,2(d-1)]$.  Suppose $X \subseteq \mathbb{R}^d$ is a bounded set such that there is a collection of lines $\mathcal{L}$ of upper box dimension $t$ such that, for all $L \in \mathcal{L}$, $  L \cap X $ is a non-empty set of upper box dimension at least  $s$.  Then
\[
\ubd X \geq \max\{s, t+1-d\}.
\]
Moreover, the same bound holds with upper box dimension replaced by lower box dimension throughout (in both assumptions and conclusion).
\end{thm}

 It is perhaps noteworthy that $s=0$ is permitted in the above result, whereas it must be omitted in the solution of the Furstenberg set problem stated above.  Indeed, a single point intersects a $(d-1)$-dimensional set of lines. Next we prove that the bounds obtained in Theorem \ref{mainbound} are sharp, thus fully resolving the Furstenberg set problem for the box dimensions in the plane and in higher dimensions.  In fact we prove something a bit stronger by allowing the intersections with lines to have \emph{Hausdorff} dimension at least $s$.

\begin{thm} \label{mainsharp}
Let  $d \geq 2$ be an integer,    $s \in [0,1]$, and $t \in [0,2(d-1)]$.  There exists a bounded set  $X \subseteq \mathbb{R}^d$ with
\[
\bd X = \max\{s, t+1-d\}
\]
and a collection of lines $\mathcal{L}$ of box dimension $t$ such that, for all $L \in \mathcal{L}$, $  L \cap X $ is a non-empty set of Hausdorff dimension at least  $s$.
\end{thm}

Next we consider the packing dimension variant.  Via a straightforward reduction (which we omit) it is possible to obtain the threshold $\max\{s,t+1-d\}$ for packing dimension as a consequence of Theorem \ref{mainbound}.  As it turns out, this bound is not sharp for $2s<t<2(d-1)$.  Our next  theorem gives stronger general bounds.

\begin{thm} \label{mainboundpacking2}
Let  $d \geq 2$ be an integer,   $s \in (0,1]$, and $t \in [0,2(d-1)]$.  Suppose $X \subseteq \mathbb{R}^d$ is a  set such that there is a collection of lines $\mathcal{L}$ of packing dimension $t$ such that, for all $L \in \mathcal{L}$, $  L \cap X $ is a set of packing dimension at least  $s$.  Then
\[
\pd  X \geq \max\{s, t/2\}.
\]
\end{thm}

Finally we prove that the bounds obtained in Theorem \ref{mainboundpacking2} are also  sharp, thus fully resolving the Furstenberg set problem for packing dimension  in the plane and in higher dimensions.  

\begin{thm} \label{mainsharppacking}
Let  $d \geq 2$ be an integer, $s \in [0,1]$, and $t \in [0,2(d-1)]$.  There exists a compact set  $X \subseteq \mathbb{R}^d$ with
\[
\pd X = \max\{s, t/2\}
\]
and a collection of lines $\mathcal{L}$ of packing dimension $t$ such that, for all $L \in \mathcal{L}$, $  L \cap X$ is a non-empty set of packing dimension  at least $s$.
\end{thm}

We remark that the problems we consider here are different from the question of estimating  the packing dimension of $(s,t)$-Furstenberg sets themselves.  This  was considered in, for example, \cite{O20, S22}.  In these works the relevant properties of $X$ were still given in terms of the \emph{Hausdorff} dimension and only the conclusion---that is, estimates for the dimension of $X$---were in terms of packing dimension.  Here we consider the `pure packing' problem where everything is in terms of packing dimension.  It may be interesting to consider `mixed cases', where one, for example, considers  the packing dimension of $X$ in terms of the \emph{Hausdorff} dimension of the line set and the \emph{packing} dimension of the intersections (or vice versa).

\subsection{Notation and convention} \label{notation}

We write $\mathcal{G}(d,1)$ for the set of all 1-dimensional subspaces of $\rd$, that is, the set of all lines through the origin.  This is a $(d-1)$-dimensional manifold equipped with the  natural Grassmannian metric given by $d(V,V') = \| P_V-P_{V'}\|$ for $V, V' \in \mathcal{G}(d,1)$.   Here $P_V$ and $P_{V'}$ denote orthogonal projection onto $V$ and $V'$ respectively, and $\| \cdot \|$ is the operator norm.  Given a set $X \subseteq \rd$, a scalar $\lambda \geq 0$, and a vector $v \in \mathbb{R}^d$, we write
\[
\lambda X + v = \{ \lambda x+v : x \in X\}.
\]
We write $\mathcal{A}(d,1)$ to denote the set  of all lines in $\mathbb{R}^d$; the \emph{affine Grassmannian}.  In particular, when we talk about \emph{a collection of lines} $\mathcal{L}$ we mean a set $\mathcal{L} \subseteq \mathcal{A}(d,1)$.  In order  to discuss the dimension of such $\mathcal{L}$ we need a metric on $\mathcal{A}(d,1)$.  For this, first observe that every line  $L \in \mathcal{A}(d,1)$ can be expressed uniquely as 
\begin{equation} \label{stanform}
L = V  + a
\end{equation}
for \emph{direction} $V \in \mathcal{G}(d,1)$   and \emph{translation}  $a \in V ^\perp$.  Since we may identify each $V ^\perp$ with $\mathbb{R}^{d-1}$, this allows us to identify  $\mathcal{A}(d,1)$ with $ \mathcal{G}(d,1) \times \mathbb{R}^{d-1}$ and equip  $\mathcal{A}(d,1)$ with the $1$-metric on  the  product $\mathcal{G}(d,1) \times \mathbb{R}^{d-1}$ where $\mathbb{R}^{d-1}$ is given the Euclidean metric, that is, 
\[
d_1(V+a, V'+a') = \|P_V-P_{V'}\| + |a-a'|.
\]
There are many similar ways to put a metric on $\mathcal{A}(d,1)$ and it is easy to see that our choice is bi-Lipschitz equivalent to the common choices found in the literature.

Throughout the paper we write $A \lesssim B$ to mean there is a constant $c\geq 1$  such that  $A \leq cB$.  Similarly we write $A \gtrsim B$ to mean $B \lesssim A$ and $A \approx B$ if both $A\lesssim B$ and $A \gtrsim B$ hold. We also write $f(k) = o(1)$ to mean that $|f(k)| \to 0$ as $k \to \infty$.

Given a bounded set $X \subseteq \mathbb{R}^d$ and a scale $\delta \in (0,1)$, write $N_\delta(X)$ to denote the smallest number of sets of diameter $\delta$ needed to cover $X$.  Then the \emph{upper} and \emph{lower box dimension} of $X$ are defined by 
\[
\ubd X = \limsup_{\delta \to 0} \frac{\log N_\delta(X)}{-\log \delta}
\]
and
\[
\lbd X = \liminf_{\delta \to 0} \frac{\log N_\delta(X)}{-\log \delta},
\]
respectively. If $\ubd X = \lbd X$ then we write $\bd X$ for the common value and call it simply the \emph{box dimension} of $X$.  The packing dimension, which is a natural dual to the Hausdorff dimension, can be defined in terms of upper box dimension. For an arbitrary non-empty set $X \subseteq \mathbb{R}^d$, the \emph{packing dimension} is
\[
\pd X = \inf \Big\{ \sup_i \ubd X_i : X \subseteq \bigcup_i X_i \Big\}.
\]
See \cite{falconer} for more details on  the box and packing dimensions.

\section{Proof of Theorem \ref{mainbound}: general bounds for box dimension variant}

We give the proof for the upper box dimension.  The analogous result for lower box dimension is proved similarly and we briefly describe this at the end. The lower bound $\ubd X \geq s$ is trivial and so it only remains to prove $\ubd X \geq t+1-d$.  We may therefore assume that $t >d-1$.  Further, assume without loss of generality that $X \subseteq B(0,1)$.  

  Given $\delta \in (0,1)$, let $\mathcal{G}_\delta$ be a minimal $\delta$-cover of  $\mathcal{G}(d,1)$ by balls of diameter $\delta$ and partition $ \mathbb{R}^{d-1}$ into a mesh $\mathcal{R}_\delta$ of $4\delta$ squares given by the product of $d-1$ intervals of the form
\[
[4N\delta, 4(N+1)\delta) \qquad (N \in \mathbb{Z}).
\]
Given a bounded collection of lines $\mathcal{L} \subseteq  \mathcal{A}(d,1)$ identified with a subset of $ \mathcal{G}(d,1) \times \mathbb{R}^{d-1}$ as described in Section \ref{notation}, let $M_\delta(\mathcal{L})$ denote the number of sets  from
\[
\{ G \times R : G \in \mathcal{G}_\delta, R \in \mathcal{R}_\delta\}
\] 
which intersect $\mathcal{L}$.  Then it is easily seen that 
\[
\ubd \mathcal{L} = \limsup_{\delta \to 0} \frac{\log M_\delta(\mathcal{L})}{-\log \delta}.
\]
Let $\eps>0$.  Since $\ubd \mathcal{L} \geq t$ by assumption, it follows that for a sequence of $\delta$ tending to zero, we must have
\begin{equation} \label{rectcover}
M_\delta(\mathcal{L}) \gtrsim \delta^{-(t-\eps)}
\end{equation}
with implicit constant independent of $\delta$ (but may depend on $\eps$).  Fix   a scale $\delta \in (0,1)$ from this sequence.  By the pigeonhole principle, there must exist $G_0 \in \mathcal{G}_\delta$ such that the number of $R \in \mathcal{R}_\delta$ for which $G_0 \times R$ intersects $\mathcal{L}$ is 
\[
\gtrsim \frac{ \delta^{-(t-\eps)}}{\# \mathcal{G}_\delta} \gtrsim \delta^{-(t+1-d-\eps)}.
\]
From each such $G_0 \times R$ which intersects $\mathcal{L}$, choose a line $L$ which intersects $X$.  From this collection of lines, if necessary, extract a subcollection of lines intersecting $X$ of cardinality
\[
\gtrsim \delta^{-(t+1-d-\eps)}
\]
 whose directions are all $\delta$ close and whose translations  are pairwise separated by $4\delta$.  For each line $L$ in this subcollection, choose a point $x \in X \cap L \subseteq B(0,1)$.  By a simple geometric argument, the points $x$ are pairwise separated by at least
\[
4\delta - 2\tan (\delta) \geq  \delta
\]
for $\delta$ sufficiently small.  This proves that $\ubd X \geq t+1-d-\eps$ and letting $\eps \to 0$ proves the desired lower bound.

The analogous result for lower box dimension can be proved in exactly the same way apart from when we choose a  sequence of $\delta$ tending to zero satisfying \eqref{rectcover} we ask for this to hold for all sufficiently small $\delta>0$.

\section{Proof of Theorem \ref{mainsharp}: sharpness  for box dimension}

We  give the proof in the case $ s \in (0,1]$ and $d-1<t\leq2(d-1)$.  The other cases are very similar (in fact, simpler) and we explain how to handle them at the end. 

We first construct $X$.  Let $E \subseteq [0,1]$ be a self-similar set satisfying the strong separation condition  with Hausdorff dimension $s$.  Alternatively $E$ can be any compact $s$-Ahlfors--David regular set in $[0,1]$. The key property we need from $E$ is that for all $c \geq 1$ and $\delta \in (0,1)$ such that $\delta c\leq 1$ the estimate
\begin{equation} \label{scaling}
N_{\delta}(c^{-1}E) = N_{\delta c}(E)  \lesssim (\delta c)^{-s}
\end{equation}
holds with implicit constants independent of $\delta$ and $c$. See \cite{falconer} for more details on self-similar sets.  Let $V_0 \in \mathcal{G}(d,1)$ be a fixed direction and $\{V_n: n \geq 1\} \subseteq \mathcal{G}(d,1)$ be a countable set with a single accumulation point at $V_0$ which satisfies
\[
\bd \{V_n\}_n = \bd \mathcal{G}(d,1) = d-1.
\]
Constructing such sets $\{V_n\}_n$ explicitly is straightforward, but we leave the details to the reader. For $V \in \mathcal{G}(d,1)$,  write  $V(E) $ for the embedding of $E$ in $V$ via an identification of $\mathbb{R}$ and $V$ as vector spaces.  (There are two choices for this identification and it does not matter which we pick.) Let
\[
Y = \bigcup_{n=1}^\infty 2^{-n} V_n(E),
\]
that is,  $Y \subseteq \rd $ is the union of copies of $E$ scaled down by a factor of $2^{-n}$ and placed in   $V_n$.   Further,  let
\[
  \{u_m : m \geq 1 \} \subseteq V _0^\perp \subseteq \rd
\]
be a bounded countable set   with
\[
\bd   \{u_m \}_m = t+1-d \in (0,d-1]
\]
and let 
\[
X = \bigcup_{m=1}^\infty (2^{-m}Y+u_m),
\]
that is, $X \subseteq \rd$ is the union of copies of $Y$  scaled down by a factor of $2^{-m}$ and translated by vectors $u_m \in V _0^\perp$.  Note that $X$ is clearly bounded.  Now, consider the  family of lines
\[
\mathcal{L} = \{ V _n + u_m \}_{m,n},
\]
noting that these lines are  \emph{not}  written in the standard form \eqref{stanform}.   By construction, for  $L =V _n + u_m \in \mathcal{L}$,
\[
L \cap X \supseteq 2^{-m-n}    V_n(E) +u_m
\]
and the right hand side is a scaled down copy of $E$ and thus has Hausdorff dimension $s$.  Moreover, although  $L$ is not written in the standard form \eqref{stanform},   the standard form representative is
\[
(V,a) = \big(V_n, P_{V_n^\perp }(u_m)\big)
\]
where $P_{V_n^\perp }$ denotes orthogonal projection onto $V _n^\perp$ which we identify  with $\mathbb{R}^{d-1}$. Therefore, as a subset of $\mathcal{G}(d,1) \times \mathbb{R}^{d-1}$, 
\[
\mathcal{L} = \left\{\big(V_n, P_{V_n^\perp}(u_m)\big) \right\}_{m,n}
\]
 and a straightforward calculation, similar to how one would  handle the Cartesian product
\[
 \left\{\left(V_n,  u_m\right) \right\}_{m,n},
\]
gives that
\[
\bd \mathcal{L} =   \bd \{ V_n    \}_n  + \bd   \{ u_m  \}_m =  (d-1)+(t+1-d)= t.
\]
This uses that $V_n \to V_0$ as $n \to \infty$ and $\{u_m\}_m \subseteq V _0^\perp$  and, therefore, $P_{V_n^\perp}(u_m) \to u_m$ uniformly as $n \to \infty$. 

All that remains is to find $\bd X$.  The lower bound $\lbd X \geq \max\{s,t+1-d\}$ follows from Theorem \ref{mainbound} and so we  prove the corresponding upper bound for upper box dimension. Let $\eps>0$ and  $\delta \in (0,1)$ be small. Define  $k(\delta)$ to be the unique integer $k \geq 1$ satisfying
\[
2^{-k} \leq \delta < 2^{-(k-1)}.
\]
 Given an integer $m < k(\delta)$ (so that $2^{-m}>\delta$), define $l(\delta,m)$ to be the unique integer $l \geq 1$ satisfying
\[
2^{-m} 2^{-l}\leq \delta < 2^{-m} 2^{-(l-1)}.
\]
Our covering strategy is as follows.  First, we construct an efficient  $\delta$-cover of $\{u_m \}_m \subseteq V _0^\perp \subseteq \mathbb{R}^d$ by balls of diameter $\delta$ in $\rd$.  This gives rise to a $\delta$-cover of comparable size of those sets $(2^{-m}Y+u_m)$ comprising $X$ whose diameters are   smaller than $\delta$.  Next we cover the remaining sets $(2^{-m}Y+u_m)$ whose diameters are bigger than $\delta$ individually.     Consider one of these sets and, without loss of generality, ignore the translation $u_m$ because it does not affect the covering number. Decompose it further as
\[
2^{-m}Y  = 2^{-m}\bigcup_{n=1}^\infty 2^{-n}   V_n(E)= \bigcup_{n=1}^\infty 2^{-m-n}   V_n(E).
\]
All but finitely many of the  sets making up the union on the right are  within $\delta$ of the origin (recall that we threw away the translation) and these are covered by $\approx 1$ many sets of diameter $\delta$.  We cover the remaining sets $2^{-m-n}   V_n(E)$ individually, and this amounts to covering a copy of $E$ scaled down by a factor which is at least $\delta$. (The cases are distinguished by $l(\delta,m)$.) Putting this argument together, we get
\begin{align*}
N_\delta(X) &\lesssim N_\delta\big(\{u_m \}_m\big) + \sum_{m=1}^{k(\delta)-1} N_\delta(2^{-m}Y+u_m) \\
&\lesssim \delta^{d-1-t-\eps} + \sum_{m=1}^{k(\delta)-1} \left( 1+ \sum_{n=1}^{l(\delta,m)-1}N_\delta(2^{-m-n}  E) \right) \\
&=\delta^{d-1-t-\eps} + k(\delta) + \sum_{m=1}^{k(\delta)-1} \sum_{n=1}^{l(\delta,m)-1} N_{\delta2^{n+m}}(E) \\
&\lesssim \delta^{d-1-t-\eps}+ \log(1/\delta) + \sum_{m=1}^{k(\delta)-1} \sum_{n=1}^{l(\delta,m)-1}  (\delta2^{n+m})^{-s} \qquad \text{(by \eqref{scaling})} \\
&\lesssim  \delta^{d-1-t-\eps}  + \delta^{-s}  \sum_{m=1}^{\infty} 2^{-ms}  \sum_{n=1}^\infty   2^{-ns}  \\
&\lesssim \delta^{d-1-t-\eps}  + \delta^{-s}.
\end{align*}
This proves $\ubd X \leq \max\{s,t+1-d+\eps\}$ and letting $\eps \to 0$ proves the desired result. 

The remaining cases are proved similarly.  If $s=0$, then $E$ can be replaced with the single point $\{0\}$ and the resulting set is simply $X = \{ u_m\}_m$ which has box dimension $t+1-d$.  If $t \leq d-1$ then setting $X=Y$ suffices, that is, replacing  $ \{ u_m\}_m$  with the single point $\{0\}$.

\section{Proof of Theorem \ref{mainboundpacking2}: general bound for packing dimension variant}

We first prove a weaker version of Theorem \ref{mainboundpacking2} where we only derive the desired conclusion for upper box dimension.

\begin{lma} \label{mainboundpacking2lma}
Let   $t \in [0,2(d-1)]$ and suppose $X \subseteq \mathbb{R}^d$ is a bounded set such that there is a collection of lines $\mathcal{L}$ of packing dimension $t$ such that, for all $L \in \mathcal{L}$, $  L \cap X $ contains at least two points.  Then
\[
\ubd  X \geq   t/2.
\]
\end{lma}

\begin{proof}
Let $\eps>0$.  For all $L \in \mathcal{L}$, there must exist an integer $n=n(L) \geq 1$ such that there are points $x,y \in L \cap X$ with $|x-y| \geq 1/n$.  Therefore,
\[
\mathcal{L} = \bigcup_{n \geq 1} \{ L \in \mathcal{L} : \text{$ \exists \, x,y \in L \cap X$ s.t.~$|x-y| \geq 1/n$}\} = : \bigcup_{n \geq 1}   \mathcal{L}_n
\]
and therefore there exists an integer $n \geq 1$ (which we fix from  now on) such that $\pd \mathcal{L}_n \geq t-\eps$.  For each $L \in \mathcal{L}_n$, identify points $L_x,L_y \in L \cap X$ such that $|L_x-L_y| \geq 1/n$. Since $\ubd \mathcal{L}_n \geq \pd \mathcal{L}_n \geq t-\eps$, we can find a decreasing  sequence of scales $\delta \in (0,1)$ tending to 0 such that there is a $\delta$-separated collection of lines in $\mathcal{L}_n $ of cardinality 
\[
\gtrsim \delta^{-(t-2\eps)}.
\]
Fix a $\delta$ from this sequence.  If
\begin{equation} \label{endsclose}
N_\delta \left( \cup_{L \in \mathcal{L}_n} \{L_x\} \right) \leq \delta^{-t/2},
\end{equation}
then, by the pigeonhole principle, there exists a single $\delta$-ball containing
\[
\gtrsim \frac{\delta^{-(t-2\eps)}}{\delta^{-t/2}} = \delta^{-(t/2-2\eps)}
\]
many of the points $L_x$. Considering the corresponding points $L_y$ and using $\delta$-separation of lines and $(1/n)$-separation of $L_x$ and $L_y$,
\[
N_\delta(X) \geq N_\delta \left( \cup_{L \in \mathcal{L}_n} \{L_y\} \right) \gtrsim  \delta^{-(t/2-2\eps)}.
\]
Here the implicit constant may depend on $n$. On the other hand, if \eqref{endsclose} does not hold, then
\[
N_\delta(X) \geq  N_\delta \left( \cup_{L \in \mathcal{L}_n} \{L_x\} \right) \geq \delta^{-t/2},
\]
and so we must have
\[
N_\delta(X) \gtrsim \delta^{-(t/2-2\eps)}.
\]
It follows  that $\ubd  X \geq  t/2-2\eps$ and letting $\eps \to 0$ proves the lemma.
\end{proof}

Next we show how to upgrade the above lemma to conclude the full theorem.  The bound $\pd  X \geq  s$ is trivial and so it remains to prove $\pd  X \geq  t/2$. Let $\eps>0$.  By definition of packing dimension we can find a decomposition 
\[
X \subseteq \bigcup_{i \in \mathbb{N}} X_i 
\]
such that each $X_i$ is bounded and 
\begin{equation} \label{packinggood}
\ubd X_i \leq \pd X + \eps
\end{equation}
for all $i \in \mathbb{N}$. Then, for all $L \in \mathcal{L}$,
\[
X \cap L \subseteq \bigcup_{i \in \mathbb{N}} X_i \cap L
\]
and therefore there must exist $i(L) \in \mathbb{N}$ such that
\[
\pd X_{i(L)} \cap L >0
\]
and, in particular, $X_{i(L)} \cap L$ contains at least two points. Then
\[
\mathcal{L} = \bigcup_{i \in \mathbb{N}}  \{ L \in \mathcal{L} :   \pd X_i \cap L >0\}
\]
and so there must exist $i \in \mathbb{N}$ such that
\[
 \pd \{ L \in \mathcal{L} : \pd X_i \cap L >0\} \geq t-\eps.
\] 
Therefore, applying Lemma \ref{mainboundpacking2lma},
\[
\ubd X_i \geq (t-\eps)/2
\]
and so, by \eqref{packinggood},
\[
\pd X \geq \ubd X_i  - \eps \geq  (t-\eps)/2 -\eps
\]
and letting $\eps \to 0$ proves the result.

\section{Proof of Theorem \ref{mainsharppacking}: sharpness  for packing dimension}

Let  $s \in (0,1]$ and $t \in (0,2(d-1)]$. If either $s=0$ or $t=0$, then trivial examples suffice. We construct the set $X$ and the line set $\mathcal{L}$ simultaneously via an iterative process indexed by $k \geq 0$.  Let $(\eta_k)_k$ be an extremely rapidly  decreasing sequence satisfying $\eta_0 = 1$, $\eta_k \to 0$ and
\[
\eta_{k+1} \leq \eta_k^k
\]
for all $k \geq 0$. In particular, this decay condition ensures $\eta_{k}= \eta_{k+1}^{o(1)}$ as $k \to \infty$. 

At step 0 in the construction, take an arbitrary line through the origin marked with the origin.  For $k \geq 0$,  the $k$th  step in the construction will consist of a non-empty $\eta_k$-separated collection of lines each marked with a  non-empty  $\eta_k$-separated collection of points.  Once these have been defined,  let $\mathcal{L}_k \subseteq \mathcal{A}(d,1)$ denote the closed $\eta_k$-neighbourhood of the lines present at step $k$ and $X_k \subseteq \mathbb{R}^d$  denote the closed $(5\eta_k)$-neighbourhood of the points present at step $k$.  Thus, $X_0 = B(0,5)$ and $\mathcal{L}_0$ is the collection of lines at distance at most 1 in $\mathcal{A}(d,1)$ from the   initial line.  It will be clear below that the sets $X_k$ and $ \mathcal{L}_k$ form decreasing nested sequences (for large enough $k$) and we set
\[
X = \bigcap_k X_k
\]
and
\[
\mathcal{L} = \bigcap_k \mathcal{L}_k
\]
which are both non-empty and compact. 

We build the collection of marked lines at step $k+1$ based on the collection at step $k$ by alternating between option (A) and option (B) detailed below.  Option (A) is designed to make $\mathcal{L}$ large and option (B) is designed to make the intersections large. 

\begin{itemize}
\item[(A)] Replace each marked line $L$ present at step $k$ with a maximal collection of $\eta_{k+1}$-separated  lines all lying in the 
\[
 \eta_{k+1}^{1-\frac{t}{2(d-1)}}\eta_k/2
\]
neighbourhood of $L$.  This produces 
\[
\approx \left( \frac{ \eta_{k+1}^{1-\frac{t}{2(d-1)}}\eta_k}{\eta_{k+1}}\right)^{2(d-1)}  =  \eta_{k+1}^{-t} \eta_k^{2(d-1)}
\]
many $\eta_{k+1}$-separated  lines replacing $L$.  For each marked point $x$ on $L$, mark each of the new lines with one point from the $\eta_{k+1}$  neighbourhood of the unique $(d-1)$-dimensional  affine hyperplane  passing through $x$ and orthogonal to $L$.  In particular, at scale $\eta_{k+1}$, the new set of marks can be covered by the number of marks on $L$ multiplied by
\[
\approx \left( \frac{\eta_{k+1}^{1-\frac{t}{2(d-1)}}\eta_k}{\eta_{k+1}}\right)^{d-1}  =  \eta_{k+1}^{-t/2} \eta_k^{d-1}
\]
and this estimate cannot be improved up to constants.

\item[(B)] Replace each marked line $L$ present at step $k$  with  the same line but with 
\[
\approx  \eta_{k+1}^{-s} \eta_k
\]
many $\eta_{k+1}$-separated marks within the $(\eta_{k+1}^{1-s}\eta_k/2)$-neighbourhood of each previous mark.
\end{itemize} 

By construction, and using the decay condition on $\eta_k$, following an application of option (A):
\[
  N_{\eta_{k+1}} (\mathcal{L}_{k+1}) \approx N_{\eta_{k}}  (\mathcal{L}_{k})   \cdot  \eta_{k+1}^{-t}\eta_k^{2(d-1)}\approx \eta_{k+1}^{-t+o(1)}
\]
and
\[
N_{\eta_{k+1}} (X_{k+1}) \approx N_{\eta_{k}} (X_{k})  \cdot  \eta_{k+1}^{-t/2} \eta_k^{d-1}\approx \eta_{k+1}^{-t/2+o(1)}
\]
and, following an application of option (B):
\[
N_{\eta_{k+1}} (\mathcal{L}_{k+1}) = N_{\eta_{k}}  (\mathcal{L}_{k})   \approx \eta_{k+1}^{o(1)}
\]
and
\[
 N_{\eta_{k+1}} (X_{k+1}) \approx N_{\eta_{k}} (X_{k}) \cdot  \eta_{k+1}^{-s} \eta_k  \approx \eta_{k+1}^{-s+o(1)}.
\]
For scales $\delta \in (0,1)$ in between terms in the sequence $(\eta_k)_k$,  say $\eta_{k+1} < \delta < \eta_k$, we get
\[
N_\delta(\mathcal{L}) \lesssim N_{\eta_k}(\mathcal{L}_k)\cdot \max\left\{\left(\frac{\eta_{k+1}^{1-\frac{t}{2(d-1)}}\eta_k}{\delta}\right)^{2(d-1)}, 1\right\}  \lesssim \delta^{-t+o(1)}
\]
and
\[
N_\delta(X) \lesssim N_{\eta_k}(X_k)\cdot \max\left\{\left(\frac{\eta_{k+1}^{1-\frac{t}{2(d-1)}}\eta_k}{\delta}\right)^{d-1},\left(\frac{\eta_{k+1}^{1-s}\eta_k}{\delta}\right), 1\right\}  \lesssim \delta^{-\max\{s, t/2\}+o(1)}.
\]
Therefore, we have established
\[
\ubd \mathcal{L} = t
\]
and 
\[
\ubd X = \max\{s, t/2\}.
\]
However, using the approximate self-similarity of the construction of  $\mathcal{L}$ and $X$, both sets have the property that their box dimensions are preserved upon (non-trivial) intersection with an arbitrary open set.  From this it follows that the upper box dimension and packing dimension coincide in both cases; see \cite[Corollary 3.10]{falconer}. 

Finally, consider the intersections.  For all $L \in \mathcal{L}$,
\[
L \cap X = \bigcap_k L \cap X_k 
\]
and, assuming we have just applied option (B) and  using the decay condition on $\eta_k$, $L \cap X_k$  contains 
\[
\approx \eta_{k}^{-s+o(1)}
\]
many pairwise disjoint intervals of length $\approx \eta_{k}$. To see this, observe that $L$ is in the $\eta_k$-neighbourhood of some line $L'$ present at the $k$th step of the above construction.  The relevant intervals in $L \cap X_k$ are the intersections of $L$ with the $(5\eta_k)$-neighbourhoods of the marked points on $L'$.   Moreover, $L \cap X_k$  is a nested sequence of sets and within each of the    intervals of length $\approx \eta_{k}$ present at step $k$, one can find 
\[
\approx \eta_{k+2}^{-s+o(1)}
\]
many pairwise disjoint intervals of length $\approx \eta_{k+2}$ at step $k+2$ (following the next application of option (B)).  Since this behaviour is repeated infinitely often, it follows that
\[
\pd L \cap X \geq  s.
\]
This completes the proof.

\end{document}